\documentclass[a4paper,reqno]{amsart}

\usepackage[utf8]{inputenc}
\usepackage[T1]{fontenc}
\usepackage{microtype}
\usepackage{tikz}
\usetikzlibrary{decorations.markings}
\usepackage{tikz-cd}

\usepackage{mathtools}
\usepackage{amsthm}
\usepackage{mathrsfs}

\usepackage[sorting=nyt,maxnames=4]{biblatex}
\usepackage{hyperref}
\hypersetup{
    colorlinks,
    linkcolor={red!50!black},
    citecolor={blue!70!black},
    urlcolor={blue!70!black}
}
\usepackage{footnotebackref}
\usepackage[capitalise]{cleveref}

\newtheorem{theorem}{Theorem}[section]
\newtheorem{proposition}[theorem]{Proposition}
\newtheorem{lemma}[theorem]{Lemma}
\newtheorem{corollary}[theorem]{Corollary}
\newtheorem*{questions}{Questions}

\theoremstyle{definition}
\newtheorem{definition}[theorem]{Definition}

\theoremstyle{remark}
\newtheorem{remark}[theorem]{Remark}
\newtheorem{example}[theorem]{Example}

\numberwithin{equation}{section}
\crefname{section}{\S\kern -3pt}{\S\S}

\DeclareMathOperator{\Ann}{Ann}
\DeclareMathOperator{\Spec}{Spec}
\DeclareMathOperator{\Tot}{Tot}
\DeclareMathOperator{\Hom}{Hom}
\DeclareMathOperator{\cone}{cone}
\DeclareMathOperator{\Ext}{Ext}
\DeclareMathOperator{\End}{End}

\DeclareMathOperator{\RHom}{\dR Hom}
\DeclareMathOperator{\Perf}{\mathfrak{Perf}}
\DeclareMathOperator{\Bl}{Bl}
\DeclareMathOperator{\gldim}{gl.dim}
\DeclareMathOperator{\Sing}{Sing}

\newcommand{\lperp}[1]{\prescript{\perp}{}{#1}} 

\newcommand{\id}{\mathrm{id}}
\newcommand{\op}{\mathrm{op}}
\newcommand{\pt}{\mathrm{pt}}
\newcommand{\dR}{\mathbf{R}}

\newcommand{\tilX}{{\widetilde X}}
\newcommand{\Z}{\mathbb{Z}}
\newcommand{\C}{\mathbb{C}}
\newcommand{\A}{\mathbb{A}}
\newcommand{\calA}{\mathcal{A}}
\newcommand{\calI}{\mathcal{I}}
\newcommand{\scrC}{\mathscr{C}}
\renewcommand{\P}{\mathbb{P}}
\renewcommand{\O}{\mathcal{O}}
\newcommand{\W}{\mathcal{W}}

\title{A categorical flop in dimension one}
\author{Calum Crossley}
\address{Department of Mathematics, University College London}
\email{calum.crossley.23@ucl.ac.uk}

\begin{document}

\begin{abstract}
    In this note we observe that the categorical structure of a flop occurs for
    some well-known non-commutative resolutions of a nodal curve. We describe
    the flop-flop spherical twists, and give a geometric interpretation in
    terms of Landau--Ginzburg models. The resolutions are all weakly crepant
    but not strongly crepant, and we formulate an intermediate condition that
    distinguishes the smaller ones.
\end{abstract}

\maketitle

\tableofcontents

\section{Introduction}

Non-uniqueness of crepant resolutions is conjectured to be addressed by an
equivalence at the level of derived categories: \cite{BO}, \cite{Kaw}. It is
then natural to ask what algebraic structure on derived categories arises from
a resolution of singularities. One approach is that of \cite{Kuz08}, where
categorical resolutions were defined in the following manner. A resolution
$\pi:\widetilde X\to X$ gives rise to adjoint functors
\begin{equation*}
    \begin{tikzcd}
        \Perf(X) \ar[d,hook] \ar[r,"\pi^*"] &
        D^b(\widetilde X) \ar[dl,"\pi_*"] \\
        D^b(X),
    \end{tikzcd}
\end{equation*}
and when $X$ has rational singularities $\pi_*\circ\pi^*$ is naturally
isomorphic to the identity on $\Perf(X)$. Kuznetsov then defined a
\emph{categorical resolution of singularities} to be a smooth triangulated
category $\scrC$ replacing $D^b(\widetilde X)$ in the above diagram, with
functors $\pi_*$ and $\pi^*$ having these same properties.

It was argued in \cite{KL15} that this definition is reasonable even when $X$
has irrational singularities, where $D^b(\widetilde X)$ actually fails to give
a categorical resolution. The authors constructed an enlargement of
$D^b(\widetilde X)$ which \emph{does} satisfy the definition, even for
non-reduced irrational singularities (in characteristic 0).

As an example, consider the case of a nodal curve $X=\{xy=0\}\subset\A^2$, with
normalization $\widetilde X=\A^1_x\amalg\A^1_y$. The categorical resolution
from \cite{KL15} is equivalent to a module category $D^b(\calA)$ over the
Auslander order $\calA=H^0\End_X(\O_X\oplus\pi_*\O_{\widetilde X})$, introduced
as a resolution for nodal and cuspidal curves in \cite{BD11}. It is well-known
that this resolution is not minimal, containing two exceptional objects which
can be contracted to give (equivalent) smaller resolutions; cf.
\cite[\S3.5]{Kuz16}, \cite[\S4.5]{LP18}.

Our main observation is that these two contractions match the behaviour of the
functors involved in a 3-fold flop, forming a categorical structure which has
already been studied for that reason (e.g. \cite[\S4]{Bar23}). Identical
structures were described for higher-dimensional nodes in
\cite[Proposition 3.15]{Kuz22}, using spinor bundles. We find the same
phenomena occuring for curves, even though the general methods break down;
categorical resolutions extend the behaviour of ``nice'' geometric resolutions
to these more degenerate irrational singularities, supporting the philosophy of
\cite{KL15}.

\subsection{Overview}

The smaller resolutions come from partial versions of $\calA$:
\begin{equation*}
    \calA_x = H^0\End_{\O_X}(\O_X\oplus\pi_*\O_{\A^1_x})
    \quad \text{and} \quad
    \calA_y = H^0\End_{\O_X}(\O_X\oplus\pi_*\O_{\A^1_y}),
\end{equation*}
over which $\calA$ is a roof via certain bimodules:
\begin{equation} \label{eqn:roof}
    \begin{tikzcd}[row sep=tiny]
        & D^b(\calA) \ar[dl] \ar[dr] & \\
        D^b(\calA_x) \ar[dr] & &
        D^b(\calA_y) \ar[dl] \\
        & D^b(X). &
    \end{tikzcd}
\end{equation}
We can then consider the push-pull composition
\begin{equation} \label{eqn:flop}
    \begin{tikzcd}[row sep=tiny]
        & D^b(\calA) \ar[dr] & \\
        D^b(\calA_x) \ar[rr,"\Phi"] \ar[ur] & &
        D^b(\calA_y),
    \end{tikzcd}
\end{equation}
and this turns out to be a derived equivalence. Composing with the reversed
version gives an autoequivalence of $D^b(\calA_x)$, which is the twist around a
spherical object.

For comparison, recall the Atiyah flop over a 3-fold node
$Y=\{uv=st\}\subset\A^4$. This has two crepant resolutions
$V_u=\Bl_{\{u=s=0\}}Y$ and $V_v=\Bl_{\{v=s=0\}}Y$, with the roof $V=\Bl_0Y$
giving a derived equivalence:
\begin{equation*}
    \begin{tikzcd}[row sep=tiny]
        & D^b(V) \ar[dr,"(\pi_v)_*"] & \\
        D^b(V_u) \ar[rr,"\simeq"] \ar[ur,"(\pi_u)^*"] & &
        D^b(V_v).
    \end{tikzcd}
\end{equation*}
The flop-flop autoequivalence on $D^b(V_u)$ is the twist around a spherical
object, given by a line bundle on the exceptional divisor.

\begin{remark}
    One interpretation of $\calA_x$, $\calA_y$ and $\calA$ is as
    ``non-commutative blowups''. For a ring $R$ and ideal $I$ one can define
    such a thing by the algebra $\End_R(R\oplus I)$, as in
    \cite[\S R]{Leu12}. We apply this to the two branches of the node and also
    their intersection, in analogy with $V_u$, $V_v$ and $V$.
\end{remark}

\begin{remark}
    As described in \cite{BKS} for the geometric situation, these diagrams of
    categories can be seen as giving examples of perverse schobers on $\C$ with
    one marked point, which categorify perverse sheaves on the hyperplane
    arrangement $(\C,0)$.
\end{remark}

\subsection{Landau--Ginzburg models}

The above resolutions are also equivalent to categories of B-branes on partial
compactifications of a Landau--Ginzburg model, from the ``exoflop''
construction of \cite{Asp15}. These LG-models were already described in the
context of mirror symmetry in \cite{LP23}, and we briefly sketch an A-side
picture of the flop in \cref{subsec:wrap}. The geometry of these LG-models
suggests a version of the flop for \emph{smooth} curves using root stacks,
which we outline in \cref{rmk:root}.

It also explains the crepancy properties of the resolutions. Kuznetsov defined
weak and strong notions of crepancy for categorical resolutions, both
equivalent to usual crepancy when $\scrC=D^b(\widetilde X)$. They are
characterized by the behaviour of a relative Serre functor on the resolution,
which for the LG-models is given by the canonical bundle. Strong crepancy asks
for the canonical bundle to be globally trivial, while weak crepancy only
requires it to be trivial on a certain open set, so categorical resolutions
from compactifications of LG-models are typically weakly crepant; see
\cite{FK17}.

In particular, all the above resolutions (including the roof) are weakly
crepant. On the other hand, they are not strongly crepant; the partial
compactifications are not Calabi--Yau. This confirms that weak crepancy is too
weak to ensure minimality of the resolution, since $D^b(\calA)$ is clearly not
minimal. Kuznetsov conjectured that strongly crepant resolutions are minimal,
although here we see that what is presumably the minimal resolution,
$D^b(\calA_x)$, fails to be strongly crepant.

Motivated by this, we introduce an intermediate condition (\cref{defn:fair}):
that the kernel subcategory $\ker(\pi_*)$ is Calabi--Yau with respect to the
relative Serre functor, i.e. the relative Serre functor restricts to a
homological shift on $\ker(\pi_*)$. This is strictly in-between weak and strong
crepancy, and seems effective at identifying minimal resolutions.
Optimistically, one could conjecture that it identifies minimal categorical
resolutions in general; see \cref{subsec:fair} for further discussion.

\subsection{Generalizations}

The algebra $\calA$ was shown to give a categorical resolution for any nodal
or cuspidal curve singularity in \cite{BD11}, and we show that a diagram with
the same properties as \cref{eqn:roof} exists for any nodal curve in
\cref{subsec:gen}. As mentioned above, the roof also has a geometric analogue
using root stacks (\cref{rmk:root}).

It would be nice to unify our results with \cite[Proposition 3.15]{Kuz22},
including for example the Atiyah flop. A preliminary remark in this direction:
$\calA$ and $\calA_x$ seem to fit into the higher Auslander--Reiten theory of
\cite[\S3]{Iya08} as 1- and 2-Auslander algebras, so a good understanding of
the results in \cite{BIKR08} may suffice to conclude that similar pictures
exist for simple curve singularities of type $A_n$ with $n$ odd and $D_m$ with
$m$ even---perhaps also in higher dimensions.

\subsection*{Acknowledgements}

I thank my supervisor Ed Segal for his invaluable guidance and encouragement,
without which this note would not have been written. I am also grateful to
Yankı Lekili for a number of helpful conversations, and to the anonymous
referee for useful suggestions of improvements to the exposition.

{\small This work was supported by the Engineering and Physical Sciences Research
Council [EP/S021590/1], the EPSRC Centre for Doctoral Training in Geometry and
Number Theory (The London School of Geometry and Number Theory), University
College London.}

\subsection*{Conventions}

We work over $\C$. Functors will be derived without changing notation. Modules
are \emph{right} modules, and quiver representations are \emph{contravariant}.

\section{Background}

We recall some of \cite{Kuz08}. Experienced readers may prefer to skip straight
to \cref{sec:flop}.

\begin{definition}
    A \emph{weak categorical resolution of singularities} for a variety $X$ is
    a (dg-enhanced) triangulated category $\scrC$ with (quasi-)functors
    \begin{equation*}
        \begin{tikzcd}
            \Perf(X) \ar[d,hook,"\iota"'] \ar[r,"\pi^*"] &
            \scrC \ar[dl,"\pi_*"] \\[-0.7em]
            D^b(X),
        \end{tikzcd}
    \end{equation*}
    where $\iota$ is the inclusion $\Perf(X)\subseteq D^b(X)$, satisfying three
    conditions:
    \begin{enumerate}
        \item \label{itm:adjoint}
            $\pi^*$ and $\pi_*$ are adjoint, in the sense that
            \begin{equation*}
                \Hom_{\scrC}(\pi^*(-),-)
                    \simeq \Hom_{D^b(X)}(\iota(-),\pi_*(-))
            \end{equation*}
            as (quasi-)functors on $\Perf(X)^\op\otimes\scrC$,

        \item $\pi_*\circ\pi^*\simeq\iota$,
            so that $\pi^*$ is fully faithful, and

        \item $\scrC$ is smooth.
    \end{enumerate}
    It is \emph{weakly crepant} if $\pi^*$ is also a \emph{right} adjoint for
    $\pi_*$ in the sense of (\ref{itm:adjoint}).
\end{definition}

\begin{remark}
    We say $\scrC$ is \emph{smooth} if $\scrC\simeq\Perf(A)$ where $A$ is a
    smooth dg-algebra, i.e. the diagonal bimodule $A$ is a perfect complex over
    $A^e=A^\op\otimes_\C A$. This is independent of choices, and equivalent to
    $\gldim{}<\infty$ if $A=H^0(A)$; \cite[\S3]{Lun10}.
\end{remark}

\begin{example}
    If $\pi:\widetilde X\to X$ is a geometric resolution of a variety $X$ with
    \emph{rational} singularities, meaning that
    $\O_X\simeq\dR\pi_*\O_{\widetilde X}$, then $D^b(\widetilde X)$ with the
    usual functors $\pi^*$ and $\pi_*$ gives a categorical resolution of $X$ by
    the projection formula. If $\pi$ is also crepant, then the categorical
    resolution is weakly crepant; the adjoints to $\pi_*$ differ iff the
    relative canonical bundle is non-trivial:
    $\pi^!=\omega_{\widetilde X/X}\otimes\pi^*$.
\end{example}

In \cite{Lun10} it is proved in broad generality that the category $D^b(X)$
is smooth. It follows that $\scrC=D^b(X)$, $\pi^*=\iota$,
$\pi_*=\id_{D^b(X)}$ is a universal weakly crepant categorical resolution
unless we make further restrictions on the behaviour of $\scrC$.

\begin{example} \label[example]{ex:koszul}
    For $X=\Spec\C[w]/w^2$ we have a strong generator
    $D^b(X)=\langle\O_X/w\rangle$, so $D^b(X)\simeq\Perf(A)$ for
    $A=\Ext^*_X(\O_X/w,\O_X/w)=\C[\theta]$ with $|\theta|=1$. This is classical
    Koszul duality. Here $A$ is smooth, but with unbounded cohomology; it fails
    to give a relatively proper resolution, being far from finite over
    $\C[w]/w^2$.
\end{example}

Various strengthenings of the definition exist. A key idea is that the category
$\scrC$ should be local on $X$, formalized by requiring $\scrC$ to be a module
for the monoidal category $\Perf(X)$---in the affine case just an
$\O_X$-linear dg-category with $(\pi^*,\pi_*)$ being $\O_X$-linear---as
introduced in \cite{Kuz08}. This is one step towards having the structure of a
non-commutative resolution \`a la Van den Bergh \cite{VdB}.

With this structure one obtains relative versions of properties of $\C$-linear
derived categories, by using the $\O_X$-module structure on $\Hom_\scrC(x,y)$.
For example $\scrC$ is proper relative to $X$ if the $\O_X$-linear chain
complex $\Hom_\scrC(x,y)$ gives an element of $D^b(X)$ for all $x,y\in\scrC$.

\begin{definition}
    Assume $X$ is affine and Gorenstein. We will call categorical resolutions
    which are $\O_X$-linear and relatively proper in the above sense
    \emph{strong categorical resolutions}. In this setting $S:\scrC\to\scrC$ is
    a \emph{relative Serre functor} for $\scrC$ if we have a natural
    isomorphism
    \begin{equation*}
        \Hom_\scrC(x,Sy) \simeq \Hom_\scrC(y,x)^\vee
            \coloneqq \Hom_{D^b(X)}(\Hom_\scrC(y,x),\O_X).
    \end{equation*}
    Kuznetsov defined $\scrC$ to be \emph{strongly crepant} if the identity is
    a relative Serre functor.
\end{definition}

\begin{remark} \label[remark]{rmk:weak}
    By \cite[Lemma 3.6]{Kuz08} we then get $S\circ\pi^*$ as a right adjoint for
    $\pi_*$, so weak crepancy of a strong categorical resolution is equivalent
    to $S$ restricting to the identity on $\pi^*(\Perf(X))$.
\end{remark}

\begin{example}
    If $\pi:\widetilde X\to X$ is a resolution of rational singularities, then
    $D^b(\widetilde X)$ is a strong categorical resolution with relative Serre
    functor $\omega_{\widetilde X/X}\otimes -$, and hence is strongly crepant
    iff $\pi$ is a crepant resolution. In this case strong and weak crepancy
    are equivalent.
\end{example}

\begin{example} \label[example]{ex:kron}
    Revisiting $X=\Spec\C[w]/w^2$, there is a strong categorical resolution
    given by the graded Kronecker quiver
    \begin{equation*}
        Q = \left(\begin{tikzcd}
            \bullet_0
                \ar[r,shift left,dashed,"\theta"] \ar[r,shift right,"w"'] &
            \bullet_1
        \end{tikzcd}\right),
    \end{equation*}
    where $|\theta|=1$.\footnote{This is smooth, being a regrading of the
    Beilinson quiver for $\P^1$ (which is smooth). It compactifies the
    $\A^1_\theta$ of \cref{ex:koszul}. Failure of strong crepancy is no
    surprise; $\P^1$ is not Calabi--Yau.} We can embed $\Perf(X)$ as the
    subcategory generated by
    \begin{equation*}
        G = \cone(P_0[-1]\xrightarrow{\theta}P_1) =
        \bigl(\begin{tikzcd}
                \C & \C \ar[l,shift left,"1"] \ar[l,shift right,dashed,"0"']
        \end{tikzcd}\bigr),
    \end{equation*}
    since $\RHom(G,G)=\C[w]/w^2$. This quiver is derived equivalent to the
    Auslander algebra $H^0\End_X(\O_X\oplus\O_X/w)$, coming from a full
    exceptional collection described in \cite[Example 5.16]{KL15} and
    \cite[\S3]{KS23}. It hence has an $\O_X$-linear structure, giving a strong
    categorical resolution with a relative Serre functor $S$ as in
    \cref{eqn:serre}. One can compute that $S(G)=G$, so the resolution is
    weakly crepant, but it is not strongly crepant since
    $\Hom(P_1,P_0)=0\ne\Hom(P_0,P_1)$.
\end{example}

\section{Flops and crepancy} \label{sec:flop}

We return to the notation from the introduction, where $X=\{xy=0\}\subset\A^2$.
To aid computations, write $\calA$ as the path algebra of a quiver with
relations:
\begin{equation*}
    \calA = \C\left(\begin{tikzcd}[row sep=small]
        & \bullet_{\A^1_x} \ar[dl,shift left,bend right,"i_x"] \\
        \bullet_X \ar[ur,shift left,bend left,"q_x"]
            \ar[dr,shift right,bend right,"q_y"'] \\
        & \bullet_{\A^1_y} \ar[ul,shift right,bend left,"i_y"']
    \end{tikzcd} \quad \middle/ \; q_xi_y = 0 = q_yi_x \right).
\end{equation*}
Then $\calA_x$ and $\calA_y$ are obtained by omitting vertex projectives:
\begin{align*}
    \calA_x &= \End_\calA(P_X\oplus P_x), &
    \calA_y &= \End_\calA(P_X\oplus P_y).
\end{align*}
The functors of \cref{eqn:roof} come from the bimodules $P_X\oplus P_x$ and
$P_X\oplus P_y$, together with $P_X$ mapping down to $D^b(X)$.

It follows that $\Perf(\calA_x)$ and $\Perf(\calA_y)$ are identified with the
subcategories
\begin{equation*}
    \lperp S_y \coloneqq \{M\in\Perf(\calA):\RHom(M,S_y)=0\}
    \quad \text{and} \quad \lperp S_x \text{ ($\coloneqq$ similar)}
\end{equation*}
orthogonal in $\Perf(\calA)$ to simple modules at the omitted vertices. These
simple modules have the following projective resolutions:
\begin{align} \label{eqn:res}
    P_y\xrightarrow{i_y}P_X\xrightarrow{q_x}P_x&\to S_x, &
    P_x\xrightarrow{i_x}P_X\xrightarrow{q_y}P_y&\to S_y,
\end{align}
from which one computes that they are exceptional objects, so we will refer to
them as $E_x\coloneqq S_x$ and $E_y\coloneqq S_y$. It follows that $\lperp E_x$
and $\lperp E_y$ are admissible subcategories of $D^b(\calA)$, defined by the
semi-orthogonal decompositions
\begin{equation*}
    D^b(\calA) = \langle E_x,\lperp E_x\rangle = \langle E_y,\lperp E_y\rangle.
\end{equation*}

\begin{proposition}
    The categories $D^b(\calA)$, $D^b(\calA_x)$ and $D^b(\calA_y)$ give strong
    categorical resolutions of $X$ with relative Serre functors.
\end{proposition}

\begin{proof}
    We have $\gldim(\calA)=2$ by \cite[Theorem 2.6]{BD11}, so
    $\Perf(\calA)=D^b(\calA)$ is smooth by \cite[Proposition 3.8]{Lun10}.
    The admissible subcategories $\Perf(\calA_{x/y})$ are also smooth by
    \cite[Theorem 3.24]{LS14}, and hence $\Perf(\calA_{x/y})=D^b(\calA_{x/y})$.
    They give $\O_X$-linear categories, which are relatively proper as
    categorical resolutions since these are coherent $\O_X$-algebras.

    Serre functors for such algebras are given by a formal construction. We
    define
    \begin{equation} \label{eqn:serre}
        S : D^b(\calA)
            \xrightarrow{\Hom_\calA(-,\calA)} D^b(\calA^\op)
            \xrightarrow{\Hom_X(-,\O_X)} D^b(\calA),
    \end{equation}
    which is a Serre functor by tensor-hom adjunction:
    \begin{align*}
        \Hom_X(\Hom_\calA(M,N),\O_X)
            &\simeq \Hom_X(N\otimes_\calA\Hom_\calA(M,\calA),\O_X) \\
            &\simeq \Hom_\calA(N,\Hom_X(\Hom_\calA(M,\calA),\O_X)).
    \end{align*}
    The same construction works over $\calA_x$ and $\calA_y$.
\end{proof}

\begin{remark} \label[remark]{rmk:serre}
    From \cref{eqn:res} we find $S(E_x)=E_y[1]$, $S(E_y)=E_x[1]$. Moreover
    \begin{equation*}
        \Hom_\calA(-,E_{x/y})^{\vee_X} = \Hom_\calA(-,E_{x/y})^{\vee_\C}[-1],
    \end{equation*}
    since $\Hom_X(\O_0,\O_X)=\O_0[-1]$, so $S[1]$ is close to being a Serre
    functor over $\C$.
\end{remark}

\subsection{Spherical twists} \label{subsec:flop}

The functor $\Phi$ of \cref{eqn:flop} is the restriction to $\lperp E_y$ of the
projection to the admissible subcategory $\lperp E_x$ in $D^b(\calA)$. It is an
equivalence:

\begin{lemma} \label[lemma]{prop:twist}
    The functor $F:D^b(\pt)\oplus D^b(\pt)\to D^b(\calA)$ mapping to the two
    exceptional objects $E_x$ and $E_y$ is spherical, and $\Phi$ is obtained by
    restricting the associated dual twist $T^{-1}:D^b(\calA)\to D^b(\calA)$ to
    $\lperp E_y$.
\end{lemma}

\begin{figure}[ht]
    \centering
    \begin{equation*}
        \begin{tikzcd}
            \bullet \ar[d,mapsto,"E_x"'] & \bullet \ar[d,mapsto,"E_y"] \\
            \bullet \ar[r,bend left,dashed,"+2"] &
            \bullet \ar[l,bend left,dashed,"+2"]
        \end{tikzcd}
    \end{equation*}
    \caption{The spherical functor $F$.}
\end{figure}
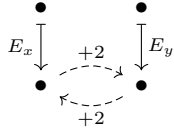

\begin{proof}
    Its right adjoint is $R=\Hom(E_x,-)\oplus\Hom(E_y,-)$, and as a consequence
    of \cref{rmk:serre} there is a left adjoint $L=RS[1]$. The cotwist
    $C=\cone(1\to RF)[-1]$ is the following bimodule over $\C\oplus\C$:
    \begin{equation*}
        \begin{bmatrix}
            0 & \Hom(E_x,E_y) \\
            \Hom(E_y,E_x) & 0
        \end{bmatrix} = \begin{bmatrix}
            0 & \C[-2] \\ \C[-2] & 0
        \end{bmatrix}.
    \end{equation*}
    This is a shift of the symmetry swapping the two points---an
    autoequivalence---and satisfies $R=CL[1]$ because
    \begin{align*}
        \Hom(E_x,-) &= \Hom(-,S(E_x))^{\vee_X} \\
                    &= \Hom(-,E_y[1])^{\vee_X} = \Hom(E_y,S(-))[-1].
    \end{align*}
    Hence $F$ is spherical by the two-out-of-four theorem
    \cite[Theorem 1.1]{AL17}.

    It follows that the twist $T=\cone(FR\to1)$ is an autoequivalence, with
    inverse the dual twist $T^{-1}=\cone(1\to FL)[-1]$. As $L=RS[1]$, this is
    given by
    \begin{equation*}
        T^{-1}(G) = \cone\biggl(G \to
            \bigl(\Hom(G,E_x)^{\vee_\C}\otimes E_x\bigr)
            \oplus\bigl(\Hom(G,E_y)^{\vee_\C}\otimes E_y\bigr)\biggr)[-1],
    \end{equation*}
    which agrees with the projection to $\lperp E_x$ when $G\in\lperp E_y$.
\end{proof}

Hence $\Phi$ is an equivalence, as $T^{-1}$ is an autoequivalence. By symmetry
the same functor $T^{-1}$ also gives the other direction of the flop. We write
$\Phi=T^{-1}$ for this extension of both functors as a slight abuse of
notation.

By \cite[Theorem 4.1.3]{Bar23} there is a simple way of expressing the
flop-flop autoequivalence $\Phi^2|_{D^b(\calA_x)}$ as a spherical twist. Since
$D^b(\calA_x)\subset D^b(\calA)$ is the complement of an exceptional object, it
is the dual twist around a spherical object: the image of $E_x$ in
$\lperp E_y=D^b(\calA_x)$. Similar holds for $D^b(\calA_y)$. Explicitly, these
spherical objects are $F_x\coloneqq\cone(E_x[-1]\to E_y[1])$ and
$F_y\coloneqq\cone(E_y[-1]\to E_x[1])$.

\begin{remark} \label[remark]{rmk:crep}
    The kernel of $D^b(\calA)\to D^b(X)$ is generated by $E_x$ and $E_y$, and
    its image in $D^b(\calA_x)$ is the kernel of $D^b(\calA_x)\to D^b(X)$,
    which is therefore just $\langle F_x\rangle$. This is in line with results
    for other nodal singularities: \cite{Sun22}, \cite{CGL23}, \cite{KS23}.

    Here $F_x$ is 3-spherical, even though $X$ is 1-dimensional, and the
    disparity means that $D^b(\calA_x)$ cannot be strongly crepant. On the
    other hand $D^b(\calA)$ and $D^b(\calA_x)$ are weakly crepant, since
    the kernels are Serre invariant; \cite[Lemma 5.8]{KS25}.
\end{remark}

\begin{remark}
    Despite being spherical in their respective subcategories, $F_x$ and $F_y$
    are not spherical in $D^b(\calA)$. Indeed, we find $S(F_x)=F_y[1]$; this is
    not a shift of $F_x$. They satisfy the sphere-cohomology part of the
    definition of a spherical object, but fail the Serre-invariance condition.
\end{remark}

\begin{remark} \label[remark]{rmk:spher}
    It turns out that $\Phi=S$. This is almost certainly a formal consequence
    of $S(E_x)=E_y[1]$ and $S(E_y)=E_x[1]$, together with the description of
    $\Phi$ in terms of mutations, but we have not been able to find a proof in
    that vein. It can be deduced indirectly from the equivalence with
    Landau--Ginzburg models outlined in \cref{sec:lg}, by matching
    \cref{rmk:div} with \cref{prop:twist}. A rather involved proof of the
    analogous claim for the nodal cubic was given in \cite[\S3]{Sun22}.

    One consequence is that the Serre functor on $D^b(\calA_x)$ is the
    restriction of $S^2$, which explains how $F_x$ becomes spherical in the
    subcategory: $S^2(F_x)=F_x[2]$.
\end{remark}

\subsection{Crepancy} \label{subsec:fair}

In view of Remarks \ref{rmk:serre} and \ref{rmk:crep}, we make the following
definition.

\begin{definition} \label[definition]{defn:fair}
    Suppose $(\scrC,\pi_*,\pi^*)$ is a strong categorical resolution of $X$,
    with a relative Serre functor $S$. We say it is \emph{fairly crepant} if
    $S(\ker(\pi_*))\subset\ker\pi_*$ and there is some $n$ such that
    $S|_{\ker(\pi_*)}=[n]$. We call $n$ the \emph{index of crepancy}.
\end{definition}

Recall that strong crepancy was the condition $S=\id_\scrC$, while weak
crepancy is the condition that $S(\ker(\pi_*))\subset\ker(\pi_*)$ by
\cite[Lemma 5.8]{KS25}, provided we revise our definition of ``strong
categorical resolution'' to also require $\pi_*$ to be a categorical
contraction \cite[Definition 5.1]{KS25}. It is then clear that strong crepancy
implies fair crepancy with index 0, and fair crepancy of any index implies weak
crepancy. The reverse implications do not hold in general.

\begin{example} \label[example]{ex:fairzero}
    Consider the categorical resolution of $X=\Spec\C[w]/w^2$ given by the
    graded Kronecker quiver from \cref{ex:kron}. The kernel is generated by
    $\cone(P_0[-d]\xrightarrow{w}P_1)$, which is a $2$-spherical object, and
    hence this resolution is fairly crepant of index $2$. For a non-trivial
    grading $|w|=d$ we can still consider categorical resolutions of
    $\C[w]/w^2$, as the definitions given above for affine schemes generalize
    immediately to graded algebras. After setting $|\theta|=1-d$, the graded
    Kronecker quiver gives a categorical resolution which is fairly crepant of
    index $2-d$. This is fairly crepant of index 0 when $d=2$, but never
    strongly crepant by the argument from \cref{ex:kron}. Hence fair crepancy
    of index 0 is weaker than strong crepancy, at least in this more general
    context. (We are not aware of a geometric example of the distinction; cf.
    Remarks \ref{rmk:nodes} and \ref{rmk:cusps}).
\end{example}

\begin{example}
    From \cref{rmk:serre} we have that $D^b(\calA)$ is weakly crepant but not
    fairly crepant, while $D^b(\calA_x)$ is fairly crepant of index 2 from
    \cref{rmk:crep} (cf. \cref{rmk:spher}). The index is the discrepancy
    between the dimension of the spherical object in the kernel (in this case
    3) and that of the base variety $X$ (in this case 1).
\end{example}

If we assume that $\scrC$ is suitably birational, so that $\Hom_\scrC(x,y)$ for
$x,y\in\ker(\pi_*)$ is a module supported at the singularity in $X$, then for
an isolated Gorenstein singularity the relative Serre functor on $\ker(\pi_*)$
is a shift of the absolute Serre functor (as in \cref{rmk:serre}). In this
situation fair crepancy of index $n$ is equivalent to requiring that
$\ker(\pi_*)$ is $(n+\dim X)$-Calabi--Yau, implying that $\ker(\pi_*)$ has no
non-trivial semi-orthogonal decompositions. The resolution then cannot be made
smaller via semi-orthogonal decomposition, assuming $\scrC$ is connected, so we
find it reasonable to expect that fairly crepant categorical resolutions do not
factor non-trivially through further categorical resolutions. Note that
suitably ``minimal'' categorical resolutions are conjectured to be contained in
all other (nice enough) categorical resolutions \cite{BO}, and in particular
should be unique, although the correct definitions to ensure this are unclear.

\begin{questions}
    \leavevmode 
    \begin{enumerate}
        \item Does the index of crepancy reduce to a known numerical invariant?
        \item In what generality do fairly crepant categorical resolutions
            exist?
        \item Are they universal minimal resolutions? In particular, are they
            unique?
    \end{enumerate}
\end{questions}

Assuming uniqueness, non-vanishing of the index of crepancy would be a
sufficient non-existence criterion for (geometric) crepant resolutions of
rational singularities, and also other strongly crepant categorical
resolutions, such as NCCRs. We find results consistent with this prediction in
Remarks \ref{rmk:nodes} and \ref{rmk:cusps}.

\subsection{General nodal curves} \label{subsec:gen}

It was shown in \cite{BD11} that $\calA$ gives a categorical resolution for any
nodal or cuspidal curve. We briefly note that exceptional objects with the same
properties as $E_x$ and $E_y$ exist in $D^b(\calA)$ for all nodal curves.

Indeed, suppose $X$ is a curve with a node at $p$, and take the normalization
$\pi:\widetilde X\to X$. Then $\pi^{-1}(p)$ consists of two reduced points
$q_1$ and $q_2$, and these give simple modules $E_1=\pi_*\O_{q_1}$ and
$E_2=\pi_*\O_{q_2}$ over $\calA=\End_X(\O_X\oplus\pi_*\O_{\widetilde X})$.

\begin{proposition}
    We have the following:
    \begin{enumerate}
        \item \label{itm:exc}
            $E_1$ and $E_2$ are exceptional objects in $D^b(\calA)$.
        \item \label{itm:serre}
            $S(E_1)=E_2[1]$ and $S(E_2)=E_1[1]$, where $S$ is as in
            \cref{eqn:serre}.
    \end{enumerate}
    In particular, the results of \cref{subsec:flop} apply to
    $\lperp E_1,\lperp E_2\subset D^b(\calA)$.
\end{proposition}

\begin{proof}
    We assume without loss of generality that $X$ is affine. The algebra
    $\calA$ has projective modules $P_X$ and $P_{\widetilde X}$, with
    \begin{equation} \label{eqn:homs}
        \begin{split}
            \End(P_X) &= \O_X, \qquad \Hom(P_X,P_\tilX) = \O_\tilX, \\
            \End(P_\tilX) &= \O_\tilX, \qquad \Hom(P_\tilX,P_X) = \calI.
        \end{split}
    \end{equation}
    Here $\calI$ is the ideal cutting out $\Sing(X)$, equal to the conductor
    $\Ann_X(\O_\tilX/\O_X)$. This is all from \cite[Proposition 2.2]{BD11}. The
    points $q_i\in\tilX$ are divisors $\{s_i=0\}$, and $g=s_1s_2$ satisfies
    $\calI=\O_\tilX\cdot g$. We then have the following resolutions:
    \begin{align*}
        &P_\tilX \xrightarrow{(-g,\,s_2)^T}
            P_X\oplus P_\tilX \xrightarrow{(1,\,s_1)} P_\tilX \to E_1, \\
        &P_\tilX \xrightarrow{(-g,\,s_1)^T}
            P_X\oplus P_\tilX \xrightarrow{(1,\,s_2)} P_\tilX \to E_2,
    \end{align*}
    with maps understood via \cref{eqn:homs}. Using these (\ref{itm:exc}) is a
    direct computation, and we also find that $(E_i)^{\vee_\calA}=E_j[-2]$,
    where $\{i,j\}=\{1,2\}$ and $E_j$ is viewed as a left-module. Then
    (\ref{itm:serre}) follows from $\Hom_X(\O_p,\O_X)=\O_p[-1]$.
\end{proof}

\begin{corollary}
    For any isolated nodal curve singularity $X$, there is a categorical
    resolution with kernel generated by a single 3-spherical object.
\end{corollary}

Such a resolution is fairly crepant of index 2, but not strongly crepant since
$2\ne0$. This extends the description of \cite{CGL23} for $\dim(X)\ge2$,
continuing the pattern of 2- or 3-spherical objects by parity of
dimension.\footnote{The fact that our resolutions are categorical contractions
in the sense of \cite{KS25} follows from \cite[Theorem 4.8]{BD11}, composing
with the admissible inclusion $D^b(\calA_x)\hookrightarrow D^b(\calA)$.} The
zero-dimensional node also follows this pattern; see \cref{ex:fairzero}.

\begin{remark} \label[remark]{rmk:nodes}
    In particular, fairly crepant resolutions exist for isolated nodal
    singularities of all dimensions. The indices of these resolutions are
    displayed in \cref{fig:index}. Note that the index is 0 in exactly the two
    dimensions where geometric crepant resolutions exist; the constructed
    categorical resolutions are equivalent to the geometric resolutions in
    these cases, and are therefore strongly crepant.
\end{remark}

\begin{table}[ht]
    \caption{Indices of crepancy for categorical resolutions of nodes.}
    \label{fig:index}
    \centering
    \begin{tabular}{|c|c|c|c|c|c|c|c|c|c|c|}
        \hline
        dim   & 0 & 1 & 2 & 3 & 4 & 5 & 6 & 7 & $\cdots$ \\
        \hline
        index & 2 & 2 & 0 & 0 & $-2$ & $-2$ & $-4$ & $-4$ & $\cdots$ \\
        \hline
    \end{tabular}
\end{table}

\begin{remark} \label[remark]{rmk:cusps}
    In \cite{Fie25}, a description of the kernels of certain categorical
    resolutions of simple hypersurface singularities of type $A_2$ is given.
    They appear to be fairly crepant of index $2-d$, where $d$ is the
    dimension.\footnote{In odd dimensions a minor extension of Fietz's results
    is required for this to be proven.} Again, this index vanishes iff a
    geometric crepant resolution exists.
\end{remark}

\section{Landau--Ginzburg models} \label{sec:lg}

Recall that a \emph{Landau--Ginzburg model} (following \cite{Seg11}), at least
on the B-side, is a smooth variety $U$ with a \emph{superpotential}
$W\in H^0(U,\O_U)$, together with an action of a rank one torus $\C^*_R$ on $U$
called the \emph{R-charge}, satisfying
\begin{enumerate}
    \item $\{\pm1\}\subset\C^*_R$ acts trivially on $U$, and
    \item $W$ has weight 2 for the $\C^*_R$ action.
\end{enumerate}
There is then an associated dg-enhanced triangulated category $D^b(U,W)$ of
\emph{matrix factorizations} of $W$, or \emph{B-branes}, whose objects are
$\C^*_R$-equivariant sheaves with an endomorphism $d$ of weight 1 satisfying
$d^2=W$, having a suitable notion of quasi-isomorphism by \cite{Orl12}.

\subsection{The exoflop}

We apply the idea of \cite{Asp15} to our situation as follows. The derived
category of $X=\{xy=0\}$ is equivalent to a category of matrix factorizations
$D^b(\A^3,xyz)$, where the R-charge weights are $|x|=|y|=0$ and $|z|=2$. This
is an instance of Kn\"orrer periodicity (\cite{Kno87}, cf. \cite{Shi12}), which
relates the derived category of a complete intersection---or more generally
factorizations on a complete intersection---to factorizations on the associated
vector bundle. Now $\widetilde X$ lies inside the blowup
$\Tot_{\P^1_{x:y}}\O(-1)$ of the plane, cut out by the section $xy$ of $\O(2)$,
and by Kn\"orrer periodicity it is also derived equivalent to an LG-model. We
get a square
\begin{equation*}
    \begin{tikzcd}
        D^b(X) \ar[r,"\simeq"] & D^b(\A^3,xyz) \\
        D^b(\widetilde X) \ar[u] \ar[r,"\simeq"] &
        D^b(\Tot_{\P^1}(\O(-1)\oplus\O(-2)),xyz), \ar[u]
    \end{tikzcd}
\end{equation*}
where $z$ is the coordinate on $\O(-2)$. This LG-model has a weighted flip:
\begin{equation*}
    \Tot_{\P^{1:1}_{x:y}}(\O(-1)_w\oplus\O(-2)_z) \dashrightarrow
        \overline U\coloneqq\Tot_{\P^{1:2}_{w:z}}\O(-1)^2_{x,y},
\end{equation*}
with $\P^{1:2}$ viewed as an orbifold. By \cite{BFK19} there is a decomposition
\begin{equation} \label{eqn:sod}
    D^b(\overline U,xyz)
        = \langle D^b(\pt),D^b(\widetilde X)\rangle,
\end{equation}
where the exceptional object is the sheafy matrix factorization
$\O_{\P^{1:2}}(-1)$ on the zero section.\footnote{By ``sheafy matrix
factorization'' we mean that the differential $d$ is zero.} Here
$D^b(\widetilde X)$ is embedded via Kn\"orrer periodicity followed by the flip.

The key observation is that $\overline U$ is a partial compactification of the
original LG-model $(\A^3,xyz)$, which is the open set $\{w\ne0\}$. Restriction
to this open set gives a functor $D^b(\overline U,xyz)\to D^b(X)$, and the
adjoint pushforward is defined for factorizations supported at $\{z=0\}$, which
corresponds to $\Perf(X)$ under Kn\"orrer periodicity. Restriction and
pushforward compose to the identity, and $D^b(\overline U,xyz)$ is smooth (from
compactification of the critical locus, or just \cref{eqn:sod}), so this is a
categorical resolution. It can be checked to be equivalent to
$D^b(\calA)$.\footnote{A suitable generator is
$\O_{\{x=0\}}\oplus\O_{\{y=0\}}\oplus\O_{\{z=0\}}$.}

\begin{remark}
    Notice how three completely different constructions---that of \cite{KL15},
    \cite{BD11}, and \cite{Asp15}---all produce the same categorical resolution
    $D^b(\calA)$. This kind of coincidence can be seen as evidence for the
    conjectural existence of a universal minimal categorical resolution, even
    though we expect there to be no strongly crepant resolution in this case.
    One slightly odd point is that the non-minimal resolution $D^b(\calA)$
    seems to crop up more naturally than the minimal one $D^b(\calA_x)$.
\end{remark}

\subsection{Removing orbifold points}

To better understand these LG-models, we should consider the critical locus of
the superpotential (\cref{fig:crit}). For $(\A^3,xyz)$ this is the node itself
in the plane $\{z=0\}$, together with a vertical $\A^1_z$ branch through the
singularity, with positive R-charge. This indicates the unbounded Ext groups
supported at the singularity of $X$; indeed, the open set $\{z\ne0\}$ of the
LG-model is equivalent by Kn\"orrer periodicity to $D^b(\A^1_z-\{0\})$, where
$|z|=2$ is positively graded. The non-trivial part of the compactification lies
in this open set, and if we apply Kn\"orrer periodicity to
$\{z\ne0\}\subset\overline U$ we instead get $D^b([\A^1_w/\Z_2])$, because the
branch $\A^1_z$ has been compactified to an orbifold $\P^{1:2}$. This
compactification of the positively graded part of the critical locus is why we
get a relatively proper categorical resolution.

A more obvious compactification just uses $\P^1$. We can achieve it
algebraically by removing an exceptional object, motivated by the decomposition
$D^b([\A^1_w/\Z_2])=\langle\O_0,\O\rangle=\langle D^b(\pt),D^b(\A^1_{w^2})\rangle$
for the chart $\{z\ne0\}$. There are two choices, since there are two twists of
the skyscraper sheaf in $[\A^1/\Z_2]$, giving two exceptional objects
$\O_{\{x=w=0\}}$ and $\O_{\{y=w=0\}}$ in $D^b(\overline U,xyz)$. Their
orthogonal complements give two smaller categorical resolutions, matching
\cref{sec:flop}.

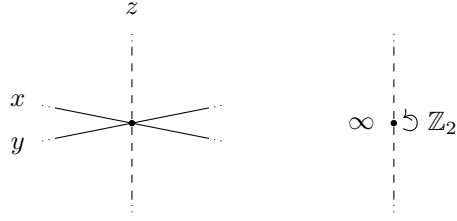
\begin{figure}[ht]
    \centering
    \begin{tikzpicture}
        \draw (-1,-0.2) -- (1,0.2);
        \draw[dotted] (1,0.2) -- (1.2,0.24);
        \draw[dotted] (-1,-0.2) -- (-1.2,-0.24);
        \node at (-1.5,-0.3) {$y$};

        \draw (-1,0.2) -- (1,-0.2);
        \draw[dotted] (1,-0.2) -- (1.2,-0.24);
        \draw[dotted] (-1,0.2) -- (-1.2,0.24);
        \node at (-1.5,0.3) {$x$};

        \draw[dashed] (0,-1) -- (0,1);
        \draw[dotted] (0,-1) -- (0,-1.2);
        \draw[dotted] (0,1) -- (0,1.2)
            node[label=above:{$z$}] {};

        \filldraw (0,0) circle (1pt);
    \end{tikzpicture} \qquad \qquad
    \begin{tikzpicture}
        \draw[dashed] (0,-1) -- (0,1);
        \draw[dotted] (0,-1) -- (0,-1.2);
        \draw[dotted] (0,1) -- (0,1.2);

        \node[label=left:{$\infty$}] at (0,0) (orb) {};
        \filldraw (orb) circle (1pt);
        \draw[->] (orb) edge[out=-30, in=30, looseness=5] (orb);
        \node[label=right:{$\Z_2$}] at (0.2,0) {};
    \end{tikzpicture}
    \caption{The critical locus of $(\A^3,xyz)$, and a chart near $z=\infty$
        for the partial compactification with an orbifold point.}
    \label{fig:crit}
\end{figure}

\begin{remark} \label[remark]{rmk:small}
    The smaller resolutions are also equivalent to factorizations on another
    LG-model: $(\Tot_{\P^1_{w:z}}(\O_x\oplus\O(-1)_y),xyz)$. Here the critical
    locus is the obvious compactification. This LG-model description appears in
    \cite{LP23}.
\end{remark}

\begin{remark} \label[remark]{rmk:root}
    These descriptions suggest the following distillation of the essence of the
    flop: for a smooth point $p$ on a curve $C$, the derived category of the
    root stack $\sqrt[2]{p/C}$ has an exceptional object, the skyscraper sheaf
    at $p$, with orthogonal complement the pullup of $D^b(C)$. It can be
    twisted by the $\Z_2$-isotropy, giving a different subcategory equivalent
    to $D^b(C)$ and a diagram like \cref{eqn:flop}. The flop-flop
    autoequivalence on $D^b(C)$ is then $-\otimes\O(p)$. We are looking at an
    analogue of this construction for the non-commutative curve $\calA_x$.

    This is an example of the general principle that root stacks behave a lot
    like blowups, especially when viewed through the lens of derived
    categories; see \cite{BLS16}.
\end{remark}

\begin{example}
    When $C=\A^1$ and $p=0$ the root stack is $[\A^1/\Z_2]$, which was the
    affine patch at infinity for the critical locus of the LG-model considered
    above.
\end{example}

\subsection{Canonical bundles}

The Serre functor on the LG-model is just given by the canonical bundle
$\O(-1)$, by a local version of \cite[Theorem 1.2]{FK17}. Weak crepancy holds
since the line bundle is trivial near $\{z=0\}$, using \cref{rmk:weak} together
with the fact that $\Perf(X)$ corresponds to factorizations supported at
$\{z=0\}$. Of course, this local condition does not imply global triviality
(i.e. strong crepancy). The orbifold point at $\{z=\infty\}$ locally obstructs
strong crepancy, contributing a local non-triviality to the canonical bundle,
but there is also the global obstruction that the toric compactifications of
$(\A^3,xyz)$ are not Calabi--Yau, so the smaller resolution is still not
strongly crepant (cf. \cref{rmk:small}).

The kernel subcategory consists of factorizations which restrict to zero on the
open set $\{z\ne0\}$, i.e. those with support at $\{z=\infty\}$, and so fair
crepancy is equivalent to local triviality of the canonical bundle near
$\{z=\infty\}$. This fails only for the orbifold compactification, where the
isotropy group gives an obstruction.

\begin{remark} \label[remark]{rmk:div}
    As a line bundle, the Serre functor is the inverse twist along the
    spherical functor given by the divisor $\{w=0\}$ with ideal sheaf $\O(-1)$.
    The source category is the restricted LG-model $D^b([\A^2/\Z_2],xy)$, which
    is equivalent by Kn\"orrer periodicity to
    $D^b([\pt/\Z_2])=D^b(\pt)\oplus D^b(\pt)$, recovering the spherical functor
    $F$.\footnote{Some care is needed for R-charge here. To invert $z$ we
    re-gauge so that $|z|=0$, giving $|x|=|y|=1$. These odd weights are
    possible because of the $\Z_2$-action.}
\end{remark}

\subsection{The A-side} \label{subsec:wrap}

Let us briefly sketch an equivalent picture under homological mirror symmetry.
For more details, the reader should consult \cite{LP18} and \cite{LP23}.

Our starting point is a standard result, that the LG-model $(\A^3,xyz)$ is
equivalent to a wrapped Fukaya category $\W(\Sigma)$, where $\Sigma$ is the
pair-of-pants surface.\footnote{In order to grade $\W(\Sigma)$ a line field
must be chosen, whose winding numbers around the three punctures of $\Sigma$
correspond to the R-charge weights of $x$, $y$ and $z$; see \cite{Seg21}.}
It was further shown in \cite{LP18} that the categorical resolution
$D^b(\calA)$ is equivalent to a \emph{partially} wrapped Fukaya category
$\W(\Sigma;\Lambda)$, where two stops $\Lambda=\{s_x,s_y\}$ have been added to
one of the boundary components (the one with positive grading). Generators
realizing the equivalence with $D^b(\calA)$ are horizontal lines in
\cref{fig:gens}. The forgetful functor $\W(\Sigma;\Lambda)\to\W(\Sigma)$ is a
localization, which contracts two exceptional objects ($E_x$ and $E_y$) given
by small arcs bounding the two stops (also depicted).

The two smaller resolutions then correspond to removing just one of these
stops, giving the equivalent categories $\W(\Sigma;s_x)$ and $\W(\Sigma;s_y)$,
with $\W(\Sigma;\Lambda)$ providing a roof via the forgetful functors. Adjoint
to these forgetful functors are the inclusions as admissible subcategories,
which can be described on objects (purely algebraically) by choosing a lift to
the roof, and then projecting to the orthogonal complement of the relevant
exceptional object via a mapping cone. This reproduces the spherical twists of
\cref{sec:flop}; the arcs near the stop in the single-stopped case are
spherical objects, and in the double-stopped case they are the exceptional
objects of \cref{prop:twist}.

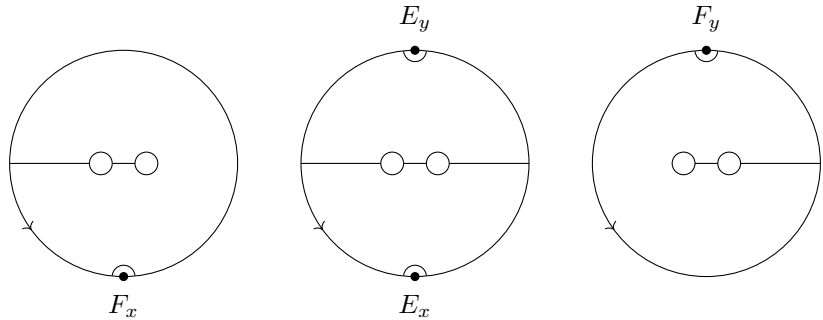
\begin{figure}[ht]
    \centering
    \begin{tikzpicture}[baseline={(0,0)},scale=1.5]
        \draw ([shift=(5:0.1)]0,-1) arc (5:175:0.1);

        \node[label=below:{$F_x$}] at (0,-1) (Lx) {};

        \draw (-1,0) -- (-0.3,0);
        \draw (-0.1,0) -- (0.1,0);

        \draw[
            decoration={markings,mark=at position 0.6 with {\arrow{>}}},
            postaction={decorate}] (0, 0) circle (1);
        \draw (-0.2,0) circle (0.1);
        \draw (0.2,0) circle (0.1);
        \filldraw (Lx) circle (1pt);
    \end{tikzpicture} \qquad
    \begin{tikzpicture}[baseline={(0,0)},scale=1.5]
        \draw ([shift=(-5:0.1)]0,1) arc (-5:-175:0.1);
        \draw ([shift=(5:0.1)]0,-1) arc (5:175:0.1);

        \node[label={$E_y$}] at (0,1) (Ly) {};
        \node[label=below:{$E_x$}] at (0,-1) (Lx) {};

        \draw (-1,0) -- (-0.3,0);
        \draw (-0.1,0) -- (0.1,0);
        \draw (0.3,0) -- (1,0);

        \draw[
            decoration={markings,mark=at position 0.6 with {\arrow{>}}},
            postaction={decorate}] (0, 0) circle (1);
        \draw (-0.2,0) circle (0.1);
        \draw (0.2,0) circle (0.1);
        \filldraw (Ly) circle (1pt);
        \filldraw (Lx) circle (1pt);
    \end{tikzpicture} \qquad
    \begin{tikzpicture}[baseline={(0,0)},scale=1.5]
        \draw ([shift=(-5:0.1)]0,1) arc (-5:-175:0.1);

        \node[label={$F_y$}] at (0,1) (Ly) {};

        \draw (-0.1,0) -- (0.1,0);
        \draw (0.3,0) -- (1,0);

        \draw[
            decoration={markings,mark=at position 0.6 with {\arrow{>}}},
            postaction={decorate}] (0, 0) circle (1);
        \draw (-0.2,0) circle (0.1);
        \draw (0.2,0) circle (0.1);
        \filldraw (Ly) circle (1pt);
    \end{tikzpicture}
    \caption{Objects in $\W(\Sigma;s_x)$, $\W(\Sigma;\Lambda)$ and
        $\W(\Sigma;s_y)$.}
    \label{fig:gens}
\end{figure}

These spherical twists are examples of the ``wrap-once'' autoequivalence for a
swappable stop from \cite{Syl19}, which is always the twist for a natural
associated spherical functor---the Orlov functor---and is defined in terms of a
``swapping'' self-isotopy of the stopping divisor. This is a full circuit
around the boundary for our single-stopped surface, but can be a half twist in
the double-stopped case (swapping the two stops), which explains why one
autoequivalence is the square of the other.

\printbibliography

\end{document}